\documentclass[12pt]{article}
\usepackage{amsmath,amsthm,amssymb,fullpage}

\newcommand{\R}{\mathbb{R}}
\newcommand{\N}{\mathbb{N}}

\newcommand{\C}{\mathbb{C}}
\newcommand{\T}{{\mathbb T}}
\newcommand{\amb}{{\R^N}} 
\newcommand{\capacity}{\textup{cap}} 
\newcommand{\card}{\textup{card}} 
\newcommand{\supp}{\textup{supp}} 
\newcommand{\fek}{\mathcal{F}} 
\newcommand{\domset}{\mathfrak{D}} 
\newcommand{\vepm}{\varepsilon_{\alpha}^{(r)}} 

\newcommand{\keywords}{\textbf{Key words. }\medskip}
\newcommand{\subjclass}{\textbf{MSC 2010. }\medskip}

\newtheorem{thm}{Theorem}[section]
\newtheorem{lem}[thm]{Lemma}
\newtheorem{prop}[thm]{Proposition}
\newtheorem{cor}[thm]{Corollary}

\theoremstyle{definition}
\newtheorem{ex}[thm]{Example}

\begin{document}

\title{\textbf{Reverse Triangle Inequalities for Riesz Potentials and Connections with Polarization}}

\author{I. E. Pritsker, E. B. Saff and W. Wise}

\date{}

\maketitle

\begin{abstract}
We study reverse triangle inequalities for Riesz potentials and their connection with polarization. This work generalizes inequalities for sup norms of products of polynomials, and reverse triangle inequalities for logarithmic potentials. The main tool used in the proofs is the representation for a power of the farthest distance function as a Riesz potential of a unit Borel measure.
\end{abstract}



\section{\textmd{Introduction: Products of Polynomials and Sums of Logarithmic Potentials}}

Let $E$ be a compact set in $\C$. For any set of real-valued functions $f_j, j=1,\ldots, m,$ we have
\begin{displaymath}
\sum_{j=1}^m \sup_E f_j \geq \sup_E \sum_{j=1}^m f_j
\end{displaymath}
by the triangle inequality. It is not possible to reverse this inequality for arbitrary functions, even by introducing additive constants. However, by restricting the class of functions we can reverse the inequality with sharp additive constants to obtain expressions of the form
\begin{equation}\label{eq:goal}
\sum_{j=1}^m \sup_E f_j \leq C+ \sup_E \sum_{j=1}^m f_j.
\end{equation}

We begin by considering logarithmic potentials $p^{\nu}(z)=\int\log|z-t|\,d\nu(t)$. Let $\nu_j$, $j=1,\ldots,m,$ be positive compactly supported Borel measures, normalized so that $\nu := \sum_{j=1}^m \nu_j$ is a unit measure. We want to find a sharp additive constant $C$ such that
\begin{equation}\label{eq:potgoal}
\sum_{j=1}^m \sup_E p^{\nu_j} \leq C+ \sup_E \sum_{j=1}^m  p^{\nu_j}=C+\sup_Ep^{\nu}.
\end{equation}

The motivation for such inequalities comes from inequalities for the norms of products of polynomials. Let $P(z) = \prod_{j=1}^n(z-a_j)$ be a monic polynomial. Then $\log|P(z)| = n\int\log|z-t|d\tau(t),$ where $\tau = \frac{1}{n}\sum_{j=1}^n\delta_{a_j}$ is the normalized counting measure of the zeros of $P$, with $\delta_{a_j}$ being the unit point mass at $a_j$. Let $||P||_E$ be the uniform (sup) norm on $E$. Then for polynomials $P_j$, $j=1,\ldots, m$, inequality \eqref{eq:potgoal} can be rewritten as
\begin{equation}\label{eq:Pgoal}
\prod_{j=1}^m || P_j ||_E \leq M^n \left|\left| \prod_{j=1}^m P_j \right|\right|_E
\end{equation}
where $M = e^{C}$ and $n$ is the degree of $\prod_{j=1}^m P_j$.

Kneser \cite{kneser} found the first sharp constant $M$ for inequality \eqref{eq:Pgoal}. Let $E=[-1, 1]$ and consider only two factors so that $m=2$. Then \eqref{eq:Pgoal} holds with the multiplicative constant
\begin{equation}\label{eq:kneser}
M = 2^{\frac{n-1}{n}} \prod_{k=1}^{\deg P_1} \left( 1+ \cos\frac{2k-1}{2n}\pi \right)^{\frac{1}{n}} \prod_{k=1}^{\deg P_2} \left( 1+ \cos\frac{2k-1}{2n}\pi \right)^{\frac{1}{n}}.
\end{equation}
The Chebyshev polynomial shows that this constant is sharp. A weaker result was previously given by Aumann \cite{aumann}. Borwein \cite{borwein} provided an alternative proof for this constant \eqref{eq:kneser}. He showed further that on $E=[-1, 1]$, inequality \eqref{eq:Pgoal} holds for any number of factors $m$ with multiplicative constant
\begin{equation}\label{eq:intervalM}
M = 2^{\frac{n-1}{n}}  \prod_{k=1}^{\left[ \frac{n}{2} \right]} \left( 1+ \cos\frac{2k-1}{2n}\pi \right)^{\frac{2}{n}}.
\end{equation}

Another series of such constants were found for $E=D$, the closed unit disk. Mahler \cite{mahler}, building on a weaker result by Gelfond \cite[p.~135]{gelfond}, showed that \eqref{eq:Pgoal} holds for
\begin{equation}\label{eq:diskM}
M = 2.
\end{equation}
While the base $2$ cannot be decreased, Kro{\'o} and Pritsker \cite{kroopritsker} showed that for $m\leq n$, we can use $M = 2^\frac{n-1}{n}$. Furthermore, Boyd \cite{boyd, boydII} expressed the multiplicative constant as a function of the number of factors $m$, and found
\begin{equation*}
M = \exp\left(\frac{m}{\pi} \int_0^{\pi/m} \log\left( 2\cos\frac{t}{2}\right) \right).
\end{equation*}
This constant is asymptotically best possible for each fixed $m$ as $n\to\infty$.

For general sets $E$, the constant $M_E$ depends on the geometry of the set. Let $E$ be a compact set of positive capacity. Pritsker \cite{pritskerspolynomials} showed that a sharp multiplicative constant in \eqref{eq:Pgoal} is given by
\begin{equation*}
M_E = \frac{\exp\left(\int \log d_E(z) d\mu_E(z) \right)}{\capacity(E)},
\end{equation*}
where $\mu_E$ is the equilibrium measure of $E$ and $d_E$ is the farthest distance function defined by $d_E(x) := \sup_{t\in E} |x-t|, \ x\in\C$. Note that this constant generalizes several previous results. We can calculate that $M_{[-1, 1]} \approx 3.20991$, which is the asymptotic version of Borwein's constant from \eqref{eq:intervalM} as $n\to\infty$. For the closed unit disk, we obtain $M_D = 2$, which is the constant given by Mahler \eqref{eq:diskM}. Furthermore, Pritsker and Ruscheweyh \cite{prI, prII} showed that $M_D$ is a lower bound on $M_E$ for any compact $E$ with positive capacity. They also conjectured that $M_{[-1, 1]}$ is an upper bound for all non-degenerate continua.  Shortly afterward, Baernstein, Pritsker, and Laugesen \cite{blp2011} showed $M_{[-1, 1]}$ is an upper bound for centrally symmetric continua. The assumption that $E$ has positive capacity is vital. For example, if $E$ is a finite set, then no inequality of the form \eqref{eq:Pgoal} is possible for any number of factors $m\geq 2$. If $E$ is countable, then the constant $M$ could grow arbitrarily fast as $m$ grows large.

All results for $M$ in \eqref{eq:Pgoal} apply as well to $C$ in \eqref{eq:potgoal} with $C = \log M$, see  \cite{reversetriangle2D}. Specifically, \eqref{eq:potgoal} holds with sharp additive constant
\begin{equation*}
C_E = \int \log d_E(z) d\mu_E(z) - \log \capacity(E).
\end{equation*}
 It follows from \cite{prI, prII} that $C_D = \log 2$ is a lower bound for $C_E$ for any compact set $E$ with positive capacity, while $C_{[-1, 1]} \approx \log 3.20991$ is an upper bound on $C_E$ for certain classes of sets $E$. Allowing the constant to be dependent on the number of terms $m$, Pritsker and Saff \cite{reversetriangle2D} found that \eqref{eq:potgoal} holds for $m$ terms with
\begin{equation*}
C_E(m) = \max_{c_k\in\partial E} \int \log \max_{1\leq k\leq m} |z-c_k| d\mu_E(z) - \log \capacity(E).
\end{equation*}
Note that $\lim_{m\to\infty} C_E(m) = C_E$.

These results were generalized to Green potentials by Pritsker \cite{reversetrianglegreen}. Let $p_j$, $j=1,\ldots,m$, be Green potentials \cite[p.~96]{armitageandgardiner} on a domain $G \subset \overline\C$. Then for any compact set $E\subset G$ we have
\begin{equation}\label{eq:gen}
\sum_{j=1}^m \inf_E p_j \geq C+M\inf_E \sum_{j=1}^m p_j
\end{equation}
where $M$ and $C$ are given in \cite{reversetrianglegreen} as explicit constants depending only on $G$ and $E$, and $C$ is sharp.

The outline of the present paper is as follows. In the next section we prove a reverse triangle inequality analogous to \eqref{eq:gen} for Riesz potentials (see Theorem \ref{thm:2.3}), and give several examples. The main ingredient in the proofs is the
representation of a power of the farthest distance function as the Riesz potential of a positive unit measure (see Theorem \ref{thm:drep}), which may be of independent interest. We consider connections of the reverse triangle inequality with polarization inequalities for Riesz potentials in Section 3. Section 4 contains all proofs.

\section{\textmd{Riesz Potentials and the Distance Function}}

We now consider a compact set $E\subset \amb$, $N\geq 2$, and Riesz potentials of the form  $U^\mu_\alpha (x) = \int |x-t|^{\alpha-N} d\mu(t)$ for $0<\alpha\leq 2$. For $\alpha = 2$, these are Newtonian potentials, and they are superharmonic in $\amb,\ N\geq 3$. If $N=\alpha=2$ then one may study inequalities for logarithmic potentials as done in \cite{reversetriangle2D}. We do not consider the case $N=\alpha=2$ in this paper. For $0<\alpha<2$, the potentials $U_\alpha^\mu$ are not superharmonic, but they are $\alpha$-superharmonic \cite[p.~111]{landkof}. Many of the standard properties of superharmonic functions hold for $\alpha$-superharmonic functions. Our goal is to find a constant $C$ such that
\begin{displaymath}
\sum_{j=1}^m \inf_E U_\alpha^{\nu_j} \geq C + \inf_E \sum_{j=1}^m U_\alpha^{\nu_j}.
\end{displaymath}
 We begin by stating some known facts. For a compact set $E \subset \amb$, let $W_\alpha(E)<\infty$ be the minimum $\alpha$-energy of $E$ and let $\mu_\alpha$ be the $\alpha$-equilibrium measure of $E$ \cite[Chapter 2]{landkof} so that
\begin{displaymath}
W_\alpha(E) = \int U_\alpha^{\mu_\alpha}d\mu_\alpha.
\end{displaymath}

\begin{thm}[Frostman's Theorem]\label{thm:Frostman}
For any compact set $E\subset \amb$ with  $W_\alpha(E)<\infty$, and any $\alpha\in (0,2]$, we have
\begin{displaymath}
U_\alpha^{\mu_\alpha}(x) \leq W_\alpha(E) , \ x\in\amb.
\end{displaymath}
Further,
\begin{displaymath}
U_\alpha^{\mu_\alpha}(x) = W_\alpha(E)  \ \text{for quasi-every } x\in E,
\end{displaymath}
where quasi-everywhere means except for a set of $\alpha$-capacity zero \cite[p.~137]{landkof}.
\end{thm}

The farthest distance function for a bounded set $E\subset\amb$ is defined by
\begin{displaymath}
d_E(x) := \sup_{t\in E} |x-t|, \ x\in\amb.
\end{displaymath}
The function $d_E$ may be expressed via potentials by using the Riesz Decomposition Theorem. The Riesz Decomposition Theorem is usually stated for Newtonian potentials \cite[p.~108]{landkof}, but a version also exists for Riesz potentials with $0<\alpha<2$ \cite[p.~117]{landkof}.

\begin{thm}\label{thm:drep}
Let $E\subset \amb$ be a compact set consisting of at least two points, with $N\geq 2$ and $0<\alpha\leq 2,\ \alpha\neq N$. Then there exists a unique positive unit Borel measure $\sigma_\alpha$ such that
\begin{equation}\label{eq:baseformula}
d^{\alpha-N}_E(x) = \int |x-t|^{\alpha-N} \ d\sigma_\alpha(t).
\end{equation}
\end{thm}
Note that if $N=\alpha=2,$ then a similar representation exists for $\log d_E(x)$, see \cite{pritskerspolynomials}. For the present paper, we are interested in expressing $d_E^{\alpha-N}$ as a potential in the case that $E= \{x_k\}_{k=1}^m \subset \amb$ is a finite set. We can give a short proof that $d_E^{2-N}$ is a potential in this case. Expressing $d_E^{2-N}$ as
\begin{displaymath}
d_E^{2-N}(x) = \left( \max_{k=1, \ldots, m} |x-x_k | \right)^{2-N} = \min_{k=1, \ldots, m} |x-x_k |^{N-2},
\end{displaymath}
we immediately see that $|x-x_k |^{N-2}$ is harmonic as a function of $x$ everywhere except at $x=x_k$ by using the Laplacian. As the minimum of harmonic functions, $d_E^{2-N}$ is superharmonic everywhere, and we may apply the Riesz Decomposition Theorem for Newtonian potentials. It states that there exists a unique Borel measure $\sigma_2$ such that
\begin{displaymath}
d^{2-N}_E(x) = \int |x-t|^{2-N} \ d\sigma_2(t) + h(x),
\end{displaymath}
where $h$ is the greatest harmonic minorant of $d_E^{2-N}$. Since $d_E^{2-N}(x)$ is positive and tends to zero as $x\to \infty$, $h$ must be identically zero \cite[p.~106]{armitageandgardiner} and hence $d^{2-N}_E$ is a Newtonian potential. The complete proof of Theorem \ref{thm:drep} is given in Section \ref{sec:Proofs}.

In the Newtonian case, we can also give explicit examples of $\sigma_2$. If $B$ is the unit ball and $d_B(x) = |x| + 1$, then the measure $\sigma_B$ is found using the Laplacian. Specifically, $d\sigma_B(x) = \Delta d_B^{2-N}(x) = c |x|^{N-2} d_B^{-N} dS$ where $c$ is a constant and $dS$ is the surface area measure on the unit sphere. If $L = [-1, 1]$ then $\Delta d_L^{2-N} = 0$ everywhere except on the hyperplane that is the perpendicular bisector of the segment. It follows that $\sigma_L$ is supported on that hyperplane \cite{me}. Its value can be calculated using the generalized Laplacian, and is given by $d\sigma_L = c d_L^{-N} dS$ where $c$ is a constant and $dS$ is the surface area measure on the hyperplane \cite{me}. For $0<\alpha<2$, the measure $\sigma_E$ should be calculated using fractional Laplacians.

We are now prepared to state a reverse triangle inequality for Riesz potentials.

\begin{thm}\label{thm:2.3}
Let $E \subset \amb$ be a compact set with the minimum $\alpha$-energy $W_\alpha(E)<\infty$, where $0<\alpha \leq 2$. Suppose that $\nu_k$, $k=1, \ldots, m$, are positive compactly supported Borel measures, normalized so that $\nu := \sum_{k=1}^m \nu_k$ is a unit measure, with $m\geq 2$. Then
\begin{equation}\label{eq:infspotentials}
\sum_{k=1}^m \inf_E U_\alpha^{\nu_k} \geq C_E(\alpha, m) + \inf_E \sum_{k=1}^m  U_\alpha^{\nu_k},
\end{equation}
where
\begin{displaymath}
C_E(\alpha, m) := \min_{c_k \in E} \int \min_{1\leq k\leq m} |x-c_k|^{\alpha-N} d\mu_\alpha(x) - W_\alpha(E)
\end{displaymath}
cannot be replaced by a larger constant for each $m\geq 2$. Furthermore, \eqref{eq:infspotentials} holds with $C_E(\alpha, m)$ replaced by
\begin{displaymath}
C_E(\alpha) := \int d_E^{\alpha-N}(x) d\mu_\alpha(x) - W_\alpha(E),
\end{displaymath}
which does not depend on $m$.
\end{thm}

In the Newtonian case $\alpha = 2$, the minimum principle holds and so the minimum in $C_E(2, m)$ is achieved on the boundary of $E$. Thus
\begin{displaymath}
C_E(2, m) = \min_{c_k \in \partial E} \int \min_{1\leq k\leq m} |x-c_k|^{2-N} d\mu_2(x) - W_2(E).
\end{displaymath}

A closed set $S \subset E$ is called \emph{dominant} if
\begin{displaymath}
d_E(x) = \max_{t\in S} |x-t| \text{ for all $x\in \supp(\mu_\alpha)$}.
\end{displaymath}
When $E$ has at least one finite dominant set, we define a \emph{minimal dominant} set $\domset_E$ as a dominant set with the smallest number of points denoted by $\card(\domset_E)$. Of course, $E$ might not have finite dominant sets at all, in which case we can take any dominant set as the minimal dominant set, e.g., $\domset_E = \partial E$.
For example, let $E$ be a polyhedron. The vertices of $E$ are a dominant set, since $d_E(x) = \max_{\text{vertices $t \in E$}} |x-t|$ everywhere, not just in $\supp(\mu_\alpha)$. However, this need not be the minimal dominant set. For example, let $E$ be a pyramid. If the apex is close to the base, then it will not be in the minimal dominant set. The hemisphere has the equator as the smallest dominant set, however this set is infinite.

\begin{cor}\label{cor:decreasingCEm}
For every $m\geq 2$, we have $C_E(\alpha, m) \geq C_E(\alpha)$. In particular, if $m< \card(\domset_E)$ then $C_E(\alpha, m) > C_E(\alpha)$, while $C_E(\alpha, m) = C_E(\alpha)$ for all $m \geq \card(\domset_E)$. Furthermore, the constants $C_E(\alpha, m)$ are decreasing in $m$ and $\lim_{m\to\infty} C_E(\alpha, m) = C_E(\alpha)$.
\end{cor}

\begin{cor}\label{cor:smoothboundary}
If $E\subset\amb$ is a compact set with $C^1$-smooth boundary and with finitely many connected components, then $C_E(\alpha, m) > C_E(\alpha)$ for all $m\in \N$, $m\geq 2$.
\end{cor}

If $E=L:=[-1,1]\subset\R^2$ and $1<\alpha<2$, then $d_L(x)=\max(|x-1|,|x+1|),\ x\in\R^2,$ so that the endpoints form the minimal dominant set with $\card(\domset_L)=2.$ Thus $C_L(\alpha)=C_L(\alpha,2)=C_L(\alpha,m),\ m\geq 2.$

We finish this section with several explicit examples.

\begin{ex}[Unit circle $\T$ in $\C$] \label{ex:circle}
Let $\T \subset \C$ be the unit circle, and let $1<\alpha<2$. We know that $d\mu_\alpha(e^{i\theta}) = d\theta/2\pi$ and
\begin{displaymath}
W_\alpha(\T) = \frac{2^{\alpha-2}}{\sqrt{\pi}} \frac{\Gamma(\frac{\alpha-1}{2})}{\Gamma(\frac{\alpha}{2})},
\end{displaymath}
see \cite{landkof}. We prove in Section \ref{sec:Proofs} that
\begin{displaymath}
\min_{c_k \in E} \int \min_{1\leq k\leq m} |x-c_k|^{\alpha-2} d\mu_\alpha(x) = 2^{\alpha-2} \frac{2m}{\pi} I\left(\frac{\pi}{2m}\right),
\end{displaymath}
where $I(x) = \int_0^x \cos^{\alpha-2}{\theta} \, d\theta$. It is obvious that $d_\T(x)=2, \ x\in \T,$ and that $S$ has no finite dominant set. Therefore,
\begin{displaymath}
C_\T(\alpha, m) = 2^{\alpha-2} \frac{2m}{\pi} I\left(\frac{\pi}{2m}\right) - \frac{2^{\alpha-2}}{\sqrt{\pi}} \frac{\Gamma(\frac{\alpha-1}{2})}{\Gamma(\frac{\alpha}{2})} > C_\T(\alpha) = 2^{\alpha-2} - \frac{2^{\alpha-2}}{\sqrt{\pi}} \frac{\Gamma(\frac{\alpha-1}{2})}{\Gamma(\frac{\alpha}{2})}.
\end{displaymath}
\end{ex}

\begin{ex}[Unit sphere $S^{N-1}$ in $\amb$] \label{ex:sphere}
Let $S^{N-1}:=\{x\in\amb: |x|=1\},\ N\geq 3,$ and let $1<\alpha\le 2$. It is known that $d\mu_\alpha = d\sigma/\omega_N$ is the normalized surface area on $S^{N-1}$ and
\begin{displaymath}
W_\alpha(S^{N-1}) = \frac{2^{\alpha-2}}{\sqrt{\pi}} \frac{\Gamma(\frac{N}{2}) \Gamma(\frac{\alpha-1}{2})}{\Gamma(\frac{N+\alpha-2}{2})},
\end{displaymath}
see \cite{landkof}. It is also clear that $d_{S^{N-1}}(x)=2, \ x\in S^{N-1},$ and that $S^{N-1}$ has no finite dominant set. Hence
\begin{displaymath}
C_{S^{N-1}}(\alpha, m) > C_{S^{N-1}}(\alpha) = 2^{\alpha-N} - \frac{2^{\alpha-2}}{\sqrt{\pi}} \frac{\Gamma(\frac{N}{2}) \Gamma(\frac{\alpha-1}{2})}{\Gamma(\frac{N+\alpha-2}{2})}.
\end{displaymath}
\end{ex}

\begin{ex}[Unit ball $B^N$ in $\amb$] \label{ex:ball}
Let $B^N:=\{x\in\amb: |x|\le 1\},\ N\geq 2,$ and let $0<\alpha\le 2$. Again, $B^N$ has no finite dominant set. The Wiener constant of the ball is
\begin{displaymath}
W_\alpha(B^N) =\frac{\Gamma(\frac{N-\alpha+2}{2}) \Gamma(\frac{\alpha}{2})}{\Gamma(\frac{N}{2})},
\end{displaymath}
see \cite{landkof}.

If $\alpha=2$ and $N\geq 3$  then the equilibrium measure of the ball $d\mu_2 = d\sigma/\omega_N$ is the normalized surface area on $S^{N-1}=\partial B^N$, so that $d_{B^N}(x)=2, \ x\in S^{N-1}=\supp(\mu_2).$ Hence
\begin{displaymath}
C_{B^N}(2, m) > C_{B^N}(2) = 2^{2-N} - 1.
\end{displaymath}

If $0<\alpha<2$ then the equilibrium measure of the ball is
\begin{displaymath}
d\mu_\alpha(x) = \frac{\Gamma\left(\frac{N-\alpha+2}{2}\right)}{\pi^{N/2}\Gamma\left(1- \frac{\alpha}{2} \right)}  \frac{R^{\alpha-N}\,dx}{(R^2-|x|^2)^{\alpha/2}} \text{ for $|x|<R$},
\end{displaymath}
see \cite[p.~163]{landkof}. Since $\supp(\mu_\alpha)=B^N$ in this case, we note that $d_{B^N}(x)=1+|x|,\ x\in B^N,$ so that
\begin{displaymath}
C_{B^N}(\alpha,m)>C_{B^N}(\alpha) = \int (1+|x|)^{\alpha-N}\, d\mu_\alpha(x) - \frac{\Gamma(\frac{N-\alpha+2}{2}) \Gamma(\frac{\alpha}{2})}{\Gamma(\frac{N}{2})},
\end{displaymath}
where $\mu_\alpha$ is given above.
\end{ex}


\section{\textmd{Connections to Polarization Inequalities}}

Let $E$ be a compact set in $\amb$ and let $A_m = \{ x_j \}_{j=1}^m$, denote an $m$-point subset of $E$. The \emph{Riesz polarization quantities}, introduced by Ohtsuka \cite{ohtsuka} and recently studied by Erd\'elyi and Saff \cite{erdsaff}, are given by
\[
M^s (A_m, E):= \inf_{x\in E} \sum_{j=1}^m |x-x_j|^{-s} \text{ and } M_m^s (E):= \sup_{A_m \subset E} M^s(A_m, E), \quad s>0.
\]
Let $\nu_j$ denote the normalized point mass $\delta_{x_j}/m$, so that $\sum_{j=1}^m \nu_j$ is a unit measure.  The Riesz polarization quantity for $s=N-\alpha$ may be rewritten in terms of potentials as
$$M^{N-\alpha}(A_m, E) = m\inf_E \sum_{j=1}^m U_\alpha^{\nu_j}.$$

 As proved by Ohtsuka \cite{ohtsuka}, the normalized limit
 \[
 \mathcal{M}^s(E) := \lim_{m\to\infty} M^s_m(E)/m
 \]
  exists as an extended real number and is called the \emph{Chebyshev constant} of $E$ for the Riesz $s$-potential.
 Moreover, he showed that this Chebyshev constant is always greater than or equal to the associated Wiener constant.
 Combining this fact with Frostman's theorem we deduce the following:
 \begin{prop} For $0<\alpha\le 2$ and any compact set $E \subset \amb$ there holds
 \begin{equation}\label{wienercheb}\mathcal{M}^{N-\alpha}(E) = W_{\alpha}(E).
 \end{equation}
 \end{prop}
Indeed, given a unit Borel measure $\mu$, Frostman's theorem for such $\alpha$ and $E$ gives
\[
\inf_E U_\alpha^{\mu} \le \int U_\alpha^{\mu}\,d\mu_{\alpha} = \int U_\alpha^{\mu_{\alpha}}\,d\mu \le W_{\alpha}(E),
\]
so that
$\mathcal{M}^{N-\alpha}(E) \leq W_\alpha(E)$, which together with Ohtsuka's inequality yields (\ref{wienercheb}). Alternatively, one can deduce (\ref{wienercheb}) by observing that for the given range of $\alpha$, a maximum
principle holds for the equilibrium potential and appealing to Theorem 11 of Farkas and Nagy \cite{farkas}.

Bounds on the quantity $M_m^{N-\alpha} (E)/m$ and the sets $A_m$ which achieve the maximum in $M_m^{N-\alpha}(E)$ have been the subject of several recent papers \cite{erdsaff, hks}. The reverse triangle inequality in Theorem \ref{thm:2.3} is directly connected with $M_m^{N-\alpha}(E)/m$ in the case of atomic measures. Recall that the inequality \eqref{eq:infspotentials} holds for arbitrary positive Borel measures $\nu_j$ such that $\sum_{j=1}^m \nu_j$ is a unit measure. We  now introduce a similar inequality where each $\nu_j=\delta_{x_j}/m$ is a point mass $1/m$ supported at $x_j\in E$:
\[
\frac{1}{m} \sum_{j=1}^m \inf_{x\in E} |x-x_j| ^{\alpha-N} \geq C_E^{\delta}(\alpha,m)
+ \frac{1}{m}  \inf_{x\in E} \sum_{j=1}^m|x-x_j| ^{\alpha-N},
\]
where $ C_E^{\delta}(\alpha,m)$ denotes the largest (best) constant such that the above inequality holds for all
$\{x_j\}_{j=1}^m \subset E$.  Clearly, we have $ C_E^{\delta}(\alpha,m) \geq  C_E(\alpha,m).$ \

From the definitions of  $C_E^{\delta}(\alpha,m)$ and
$M_m^{N-\alpha}(E)$ we immediately deduce that for all $\alpha < N$,
\[
\max_{A_m \subset E} \frac{1}{m} \sum_{j=1}^m d_E^{\alpha-N}(x_j) -\frac{M_m^{N-\alpha}(E)}{m} \geq C_E^{\delta}(\alpha,m).
\]
In particular, if $E$ is the unit sphere $S^{N-1}\subset \mathbb{R}^N$, we have
\begin{equation}\label{sphere}
2^{\alpha-N}-\frac{M_m^{N-\alpha}(S^{N-1})}{m}=C_{S^{N-1}}^{\delta}(\alpha,m).
\end{equation}
In \cite{hks}, it is proved that for the unit circle ${\T}=S^1$ the maximum polarization for any
$m\ge 2$ is attained for $m$ distinct equally spaced points. Moreover, this maximum, which occurs at the midpoints of the $m$ subarcs joining adjacent points is known explicitly (in finite terms) when $N-\alpha$ is a positive even integer, and asymptotically for all $-\infty <\alpha<N$.  Thereby we obtain the following.

\begin{prop}\label{circleprop}For the unit circle ${\T}=S^1$ there holds, for all $-\infty<\alpha <2,$
\begin{equation}\label{main}
C_{\T}^{\delta}(\alpha,m)=2^{\alpha-2}-\frac{M^{2-\alpha} (A^*_m, \T)}{m}=2^{\alpha-2}-\frac{M_m^{2-\alpha} ( \T)}{m},
\end{equation}
where $A^*_m=\{e^{i2\pi k/m} : k=1,\ldots, m\}.$  Moreover the following asymptotic formulas hold
as $m\to \infty:$

\begin{equation}\label{asym}
C_{\T}^{\delta}(\alpha,m) \sim
\begin{cases}
 \displaystyle{-\frac{2 \zeta(2-\alpha)}{(2\pi)^{2-\alpha}} \, (2^{2-\alpha} - 1)m^{1-\alpha}} \,, \quad
\enskip 1 > \alpha >-\infty\,, \\
\displaystyle {-\frac{1}{\pi} \, \log m}\,, \quad   \enskip \alpha = 1\,, \\
\displaystyle{2^{\alpha-2}-\frac{2^{\alpha-2}}
{\sqrt{\pi}} \, \frac{\Gamma \big( \frac{\alpha-1}{2}\big)}{\Gamma \big(\frac{\alpha}{2} \big)}}= C_{\T}(\alpha)\, ,
\quad 1<\alpha<2,
\end{cases}
\end{equation}
where $\zeta(s)$ denotes the classical Riemann zeta function and
 $a_m \sim b_m$ means that $\lim_{m\rightarrow \infty}{a_m/b_m} = 1$.
\end{prop}
For $1<\alpha<2$, we have from Example 2.6 and \eqref{asym} that, for each $m \geq 1,$
\[
C_{\T}(\alpha)<C_{\T}(\alpha,m) \leq C_{\T}^{\delta}(\alpha,m),
\]
with equality holding throughout in the limit as $m \to \infty.$ Consequently, from the formulas in Example \ref{ex:circle} we have
$$\frac{M_m^{N-\alpha}(\T)}{m}=2^{\alpha-2}-C_{\T}^{\delta}(\alpha,m) \le2^{\alpha-2}-C_{\T}(\alpha,m)=W_\alpha(\T) + 2^{\alpha-2} \left( 1 - \frac{2m}{\pi} I\left( \frac{\pi}{2m} \right) \right) < W_\alpha(\T).$$
We remark that the inequality $M_m^{N-\alpha}(\T) \le m W_\alpha(\T)$ was found by a different method in (3.7) of \cite{erdsaff}.

Utilizing \eqref{main} and the polarization formulas in \cite{hks}, we list the first few  explicit formulas for $C_{\T}^{\delta}(\alpha,m)$ that hold whenever $\alpha$ is a  nonpositive even integer and $m \geq 1$:

$$C_{\T}^{\delta}(0,m)=\frac{1}{4}-\frac{m}{4},
$$
$$
C_{\T}^{\delta}(-2,m)=\frac{1}{16}-\frac{m}{24}-\frac{m^3}{48},
$$
$$C_{\T}^{\delta}(-4,m)=\frac{1}{64}-\frac{m}{120}-\frac{m^3}{192}-\frac{m^5}{480}.
$$\\
For the unit sphere in higher dimensions, we have the following.

\begin{prop}\label{sphereprop}
For the unit sphere $S^{N-1}, N>2,$ in $\mathbb{R}^N$ equation \eqref{sphere} holds for all $-\infty<\alpha <N.$ Moreover, the following asymptotic formulas hold
as $m\to \infty:$

\begin{equation}\label{asymsphere}
C_{S^{N-1}}^{\delta}(\alpha,m) \sim
\begin{cases}
 \displaystyle{-\sigma(N-\alpha,N-1)\left(\frac {\Gamma(N/2)}{2\pi^{N/2}}\right)^{(N-\alpha)/(N-1)}m^{\frac{1-\alpha}{N-1}}} \,, \quad
\enskip 1 > \alpha >-\infty\,, \\
\displaystyle {-\frac{\log m}{\sqrt{\pi}} \frac{\Gamma(N/2)}{(N-1)\Gamma((N-1)/2)}}\,, \quad   \enskip \alpha = 1\,, \\
\displaystyle{2^{\alpha-N}-W_{\alpha}(S^{N-1})= C_{S^{N-1}}(\alpha)\,} ,
\quad 1<\alpha<N,
\end{cases}
\end{equation}
where $\sigma(N-\alpha,N-1)$  is a positive constant that depends only on $\alpha$ and $N$ (cf. \cite{BHS4}), and where the formulas for  $W_{\alpha}(S^{N-1})$ and $C_{S^{N-1}}(\alpha)$ are given in Example 2.7.
\end{prop}

For the unit ball we have the following result.

\begin{prop}\label{ballprop}
For the unit ball $B^N$ in $\mathbb{R}^N$ there holds, for all $-\infty<\alpha <N,$
\begin{equation}\label{mainball}
1-\frac{M_m^{N-\alpha} ( B^N)}{m}\geq C_{B^N}^{\delta}(\alpha,m)\geq 2^{\alpha-N}-\frac{M_m^{N-\alpha} ( B^N)}{m}.
\end{equation}
 Moreover the following asymptotic formulas hold as $m\to \infty:$

\begin{equation}\label{asymball}
C_{B^N}^{\delta}(\alpha,m) \sim
\begin{cases}
 \displaystyle{-\sigma(N-\alpha ,N )\left(\frac{\Gamma(1+N/2)}{\pi^{N/2}} \right)^{(N-\alpha)/N}m^{-\alpha/N}} \,, \quad
\enskip 0 > \alpha >-\infty\,, \\
\displaystyle {- \log m}\,, \quad   \enskip \alpha = 0\,, \\
\end{cases}
\end{equation}
where $\sigma(N-\alpha,N)$ is a positive constant that depends only on $\alpha$ and $N$.
\end{prop}

We remark that asymptotic formulas similar to those in Proposition \ref{ballprop} can be obtained for  $C_E^{\delta}(\alpha,m)$
for a large class of $N$-dimensional subsets of $\mathbb{R}^N$ by appealing to the results in \cite{borod} and \cite{BHS4}.

\section{\textmd{Proofs}}\label{sec:Proofs}

We begin with a lemma that will be used in the proof of Theorem \ref{thm:2.3}. Let $F_n = \{x_{k, n}\}_{k=1}^n$ be a set of $n$ points in $E$. Let $\tau_n$ be their normalized counting measure and let $0<\alpha<N$. We define the \emph{discrete $\alpha$-energy} of $\tau_n$ by
\begin{displaymath}
E_\alpha[\tau_{n}] := \frac{2}{n(n-1)} \sum_{1\leq j<k\leq n} |x_{j, n}-x_{k, n}|^{\alpha-N}.
\end{displaymath}
As $E$ is compact, the minimum discrete $\alpha$-energy is achieved by some set of points. Let $\fek_n = \{\xi_{k, n}\}_{k=1}^n$, be a set of $n$ points in $E$ that minimizes the discrete $\alpha$-energy. For $\alpha = 2$, these are typically called the Fekete points. They provide a way to approximate the $\alpha$-equilibrium measure.

\begin{lem}\label{lem:liminfFeketepts}
Given $0<\alpha< N$, let $\fek_n:=\{\xi_{k, n}\}_{k=1}^n$ be the points of $E$ minimizing the discrete $\alpha$-energy. Let $\tau_n$ be the normalized counting measure associated with the set $\fek_n$. Then the discrete $\alpha$-energies of the measures $\tau_n$ increase monotonically and converge weak$^*$ to the $\alpha$-equilibrium measure $\mu_\alpha$. Further,
\begin{displaymath}
\lim_{n\to\infty}\inf_E U_\alpha^{\tau_n} = \lim_{n\to\infty}\inf_{x\in E} \frac{1}{n}\sum_{k=1}^n |x-\xi_{k, n}|^{\alpha-N} = W_\alpha(E).
\end{displaymath}
\end{lem}
\begin{proof}


The facts that the discrete energies of the measures $\tau_n$ increase monotonically and converge weak$^*$ to the equilibrium measure are proved in \cite[p.~160-162]{landkof}. Since $\tau_n$ is a unit measure, we may apply Tonelli's Theorem followed by Frostman's Theorem \ref{thm:Frostman} to find
\begin{displaymath}
\int U^{\tau_n}_\alpha d\mu_\alpha = \int U^{\mu_\alpha}_\alpha d\tau_n \leq W_\alpha(E).
\end{displaymath}
Since $\supp(\mu_\alpha) \subset E$, this implies
\begin{displaymath}
\inf_E U_\alpha^{\tau_n} \leq W_\alpha(E).
\end{displaymath}
On the other hand, for the $(n+1)$-tuple $(x, \xi_{1, n}, \ldots, \xi_{n, n}) \subset E$ we may again apply the extremal property of $\fek_n$ to obtain
\begin{displaymath}
\sum_{1\leq j < k \leq n+1} |\xi_{j, n+1}- \xi_{k, n+1}|^{\alpha-N} \leq
\sum_{k=1}^n |x - \xi_{k, n}|^{\alpha-N}
+ \sum_{1\leq j < k \leq n} |\xi_{j, n}- \xi_{k, n}|^{\alpha-N}.
\end{displaymath}
Further, monotonicity of discrete energies gives that
\begin{align*}
\sum_{k=1}^n |x - \xi_{k, n}|^{\alpha-N} &\geq \frac{n(n+1)}{n(n+1)}\sum_{1\leq j < k \leq n+1} |\xi_{j, n+1}- \xi_{k, n+1}|^{\alpha-N} -  \sum_{1\leq j < k \leq n} |\xi_{j, n}- \xi_{k, n}|^{\alpha-N} \\
&\geq \frac{n(n+1)}{n(n-1)}\sum_{1\leq j < k \leq n} |\xi_{j, n}- \xi_{k, n}|^{\alpha-N} -  \sum_{1\leq j < k \leq n} |\xi_{j, n}- \xi_{k, n}|^{\alpha-N} \\
&= \frac{2}{(n-1)}\sum_{1\leq j < k \leq n} |\xi_{j, n}- \xi_{k, n}|^{\alpha-N},
\end{align*}
which immediately implies that
\begin{displaymath}
W_\alpha(E) \geq \inf_{x\in E} \frac{1}{n} \sum_{k=1}^n |x-\xi_{k, n}|^{\alpha-N} \geq E_\alpha[\tau_n]\to W_\alpha(E)  \text{ as } n\to\infty.
\end{displaymath}
\end{proof}

\begin{proof}[Proof of Theorem \ref{thm:drep}]
As discussed following the statement of Theorem \ref{thm:drep}, if $E$ is finite and $\alpha = 2$, then $d^{2-N}_E$ is the Newtonian potential of a unique Borel measure $\sigma_2$. A similar proof will show that, for $0<\alpha<2$, $d^{\alpha-N}_E$ is the  Riesz $\alpha$-potential of a unique Borel measure. Note that for each $v\in E$, $|x-v|^{\alpha-N}$ is a Riesz kernel and thus is $\alpha$-superharmonic \cite[p.~113]{landkof}. Since $d_E^{\alpha-N}(x) = \min_{v\in E} |x-v|^{\alpha-N}$ is a minimum of finitely many $\alpha$-superharmonic functions, it is also $\alpha$-superharmonic \cite[p.~129]{landkof}. Hence there exists a unique positive Borel measure $\sigma_\alpha$, such that $d_E^{\alpha-N}(x) = U^{\sigma_\alpha}_\alpha(x) + A$ for some constant $A\geq 0$ \cite[p.~117]{landkof}. Since  $d_E^{\alpha-N}(x)$ tends to zero as $x\to\infty$ and $U^{\sigma_\alpha}_\alpha(x)$ is positive everywhere, we conclude that $A = 0$. Thus $d_E^{\alpha-N}(x) = U^{\sigma_\alpha}_\alpha(x)$ for all $x\in\amb.$

If $E$ is compact, we consider a sequence of finite subsets $E_m \subset E_{m+1} \subset E$ that are dense in $E$ as $m\to\infty$. Let $d_m$ be the farthest distance function of $E_m$ and let $\sigma_m$ be the associated measure such that $d_m^{\alpha-N} = U^{\sigma_m}_\alpha,\ m\in\N$. Since $d_m\leq d_{m+1}$, it follows that $U^{\sigma_m}_\alpha \geq U^{\sigma_{m+1}}_\alpha,\ m\in\N$. Thus we obtain a decreasing sequence of potentials, and  Theorem 3.10 of \cite{landkof} gives a positive unique Borel measure $\sigma_\alpha$ such that $\sigma_m \stackrel{*}{\rightarrow} \sigma_\alpha$ and $d_E^{\alpha-N} = U^{\sigma_\alpha}_\alpha$ quasi-everywhere. Since the set of points $S$ where $d_E^{\alpha-N} \neq U^{\sigma_\alpha}_\alpha$ has $\alpha$-capacity zero, it also has zero volume in $\amb$, see \cite[Theorem 3.13 on p.~196]{landkof}. Hence $U^{\sigma_\alpha}_\alpha \ast \vepm = d_E^{\alpha-N} \ast \vepm$ for the averaging  measure $\vepm$ used in the definition of $\alpha$-superharmonicity in \cite[p.~112]{landkof}. Furthermore, Property (i) \cite[p.~114]{landkof} for $\alpha$-superharmonic functions gives that
\begin{displaymath}
U^{\sigma_\alpha}_\alpha(x) = \lim_{r\to 0} U^{\sigma_\alpha}_\alpha \ast \vepm (x) = \lim_{r\to 0} d_E^{\alpha-N}\ast \vepm (x) = d_E^{\alpha-N}(x), \quad x\in\amb,
\end{displaymath}
where we used the fact that $\vepm \stackrel{*}{\rightarrow} \delta_0$ as $r\to 0$ \cite[p.~112]{landkof} on the last step.

We now consider the mass of the measure $\sigma_\alpha$. Assume, without loss of generality, that the origin is a point in $E$. Consider the ball $B(R)$ of radius $R>\text{diam}(E)$ about the origin. We average
\begin{equation*}
d_E^{\alpha-N}(x) = \int_\amb  |x-t|^{\alpha-N} d\sigma_\alpha(t)
\end{equation*}
with respect to the $\alpha$-equilibrium measure $\tau_R$ of the ball $B(R)$, to obtain
\begin{equation}\label{averagetogetunitmeasure}
M(R) := \int_{B(R)} d_E^{\alpha-N}(x)  d\tau_R(x) =  \int_{B(R)}  \int_\amb  |x-y|^{\alpha-N} d\sigma_\alpha(y) d\tau_R(x).
\end{equation}

We  begin with the case $\alpha = 2$, for which the $\alpha$-equilibrium measure is the normalized surface area measure $d\tau_R =  dx/(\omega_N R^{N-1})$. It is a standard fact \cite[p.~100]{armitageandgardiner} that the potential of the equilibrium measure is given by
\begin{equation*}
U_2^{\tau_R}(x) = \frac{1}{\omega_N R^{N-1}} \int |x-t|^{2-N} dt = \begin{cases}
R^{2-N} & \text{if $|x|\leq R$},\\
|x|^{2-N} & \text{if $|x|> R$}.
\end{cases}
\end{equation*}
Consider the left hand side of \eqref{averagetogetunitmeasure}. We know $R \leq d_E(x) \leq R + \text{diam}(E)$ on  $\supp(\tau_R) = \partial B(R)$. Thus
\begin{equation}\label{eq:mean}
(R + \text{diam}(E))^{2-N} \leq M(R) \leq R^{2-N}.
\end{equation}
On the other hand, we may apply Tonelli's Theorem to the right hand side of \eqref{averagetogetunitmeasure} and obtain
\begin{align*}
\int_{\partial B(R)}  \int_\amb  |x-t|^{2-N} d\sigma_2(t) d \tau_R(x)  &=  \int_\amb \int_{\partial B(R)}    |x-t|^{2-N} d \tau_R(x) d\sigma_2(t) \\
& = \int_\amb U_2^{\tau_R}(y) d\sigma_2(y) \\
& = \int_{|t|<R} U_2^{\tau_R}(t) d\sigma_2(t) + \int_{|t|>R} U_2^{\tau_R}(t) d\sigma_2(t) \\
& = R^{2-N} \sigma_2(B(R)) + \int_{|t|>R} U_2^{\tau_R}(t) d\sigma_2(t).
\end{align*}
Since $0 < U_2^{\tau_R}(t) < R^{2-N}$ for $|t|>R$, we have
\begin{displaymath}
R^{2-N} \sigma_2(B(R)) \leq M(R) \leq R^{2-N} \sigma_2(\amb).
\end{displaymath}
Combining the above inequality with \eqref{eq:mean} and letting $R\to\infty$, we obtain $\sigma_2(\amb) = 1$.

The proof in the case of $0<\alpha<2$ is similar. In this case, the $\alpha$-equilibrium measure is given in \cite[p.~163]{landkof} as
\begin{displaymath}
d\tau_R(x) = A  R^{\alpha-N} (R^2-|x|^2)^{-\alpha/2}dx \text{ for $|x|<R$},
\end{displaymath}
where $A$ is the constant
\begin{displaymath}
A = \frac{\Gamma\left(\frac{N-\alpha}{2} + 1 \right)}{\pi^{N/2}\Gamma\left(1- \frac{\alpha}{2} \right)}.
\end{displaymath}
Its potential $U_\alpha^{\tau_R}(x) = \int |x-t|^{\alpha-N} d\tau_R(t)$ is
\begin{displaymath}
U_\alpha^{\tau_R}(y) = A  R^{\alpha-N} \frac{\pi^{N/2+1}}{\Gamma(N/2) \sin(\pi\alpha/2)},
\end{displaymath}
for all $|y|\le R$ \cite[(A.1)]{landkof}. Using the fact that
\begin{displaymath}
\frac{\pi}{\sin(\pi x)} = \Gamma(x)\Gamma(1-x),
\end{displaymath}
we calculate for $|y|\le R$ that
\begin{align*}
U_\alpha^{\tau_R}(y) =&  \frac{\Gamma\left((N-\alpha)/2 + 1 \right)}{\pi^{N/2}\Gamma\left(1- \alpha/2 \right)} R^{\alpha-N} \frac{\pi^{N/2}}{\Gamma(N/2)} \Gamma(\alpha/2)\Gamma(1-\alpha/2) \\
=&   \frac{\Gamma(\alpha/2) \Gamma\left((N-\alpha)/2 + 1 \right)}{\Gamma(N/2)} \ R^{\alpha-N}.
\end{align*}
Introducing the notation
\begin{displaymath}
c(N, \alpha) = A \frac{\pi^{N/2+1}}{\Gamma(N/2) \sin(\pi\alpha/2)} = \frac{\Gamma(\alpha/2) \Gamma\left((N-\alpha)/2 + 1 \right)}{\Gamma(N/2)},
\end{displaymath}
we obtain
\begin{equation}\label{eq:apot}
U_\alpha^{\tau_R}(y) = c(N, \alpha) R^{\alpha-N}, \quad \text{for all $|y|\leq R$.}
\end{equation}
Furthermore, this same value serves as the upper bound of the potential for all $|y|>R$. Notice that $c(N, 2) = 1$ and hence \eqref{eq:apot} is a generalization of the fact that $U_2^{\tau_R}(x) = R^{2-N}$ for $|x|\leq R$ when $\alpha=2$.

Consider the left hand side of \eqref{averagetogetunitmeasure}. We know $|x| \leq d_E(x) \leq |x| + \text{diam}(E)$ in $B(R)$. We use the lower bound on $d_E$ to find an upper bound on $M(R)$. Applying the calculations in \cite[Appendix]{landkof} again, we conclude that
\begin{align*}
\int_{B(R)} d_E^{\alpha-N}(x)  d\tau_R(x) &\leq \int_{B(R)} |x|^{\alpha-N}  d\tau_R(x) \\
&= A R^{\alpha-N} \int_{B(R)} (R^2-|x|^2)^{-\alpha/2} |x|^{\alpha-N}  dx \\
&= A  R^{\alpha-N} \frac{\pi^{N/2+1}}{\Gamma(N/2) \sin(\pi\alpha/2)} \\
&=  c(N, \alpha) R^{\alpha-N}.
\end{align*}
Next we use the upper bound on $d_E$ to obtain a lower bound for $M(R)$. Let $d=\text{diam}(E)$. Then for any $\epsilon>0$ we have $d\le \epsilon |x|$ for any $x$ not in $B(d/\epsilon)$. Hence
\begin{align*}
\int_{B(R)} d_E^{\alpha-N}(x)  d\tau_R(x) &\geq \int_{B(R)} (|x|+d)^{\alpha-N}  d\tau_R(x) \\
&> \int_{B(R) \setminus B(d/\epsilon)} (|x|+d)^{\alpha-N}  d\tau_R(x) \\
&\geq \int_{B(R) \setminus B(d/\epsilon)} |x|^{\alpha-N} \left(1+\epsilon \right)^{\alpha-N}  d\tau_R(x) \\
&=\left(1+\epsilon \right)^{\alpha-N} U_\alpha^{\tau_R}(0) -  \left(1+\epsilon \right)^{\alpha-N} \int_{B(d/\epsilon)} |x|^{\alpha-N}  d\tau_R(x) \\
&= \left( 1+\epsilon \right)^{\alpha-N}  c(N, \alpha) R^{\alpha-N} -  \left(1+\epsilon \right)^{\alpha-N} \int_{B(d/\epsilon)} |x|^{\alpha-N}  d\tau_R(x)
\end{align*}
Estimating the integral over the ball $B(d/\epsilon)$, we find
\begin{align*}
\int_{B(d/\epsilon)} |x|^{\alpha-N}  d\tau_R(x) =& R^{\alpha-N}A\omega_N \int^{d/\epsilon}_0 |x|^{\alpha-1}  (R^2-|x|^2)^{-\alpha/2} d|x| \\
\leq& R^{\alpha-N} \left(R^2-\frac{d^2}{\epsilon^2}\right)^{-\alpha/2} \frac{A\omega_N d^\alpha}{\alpha\epsilon^\alpha}.
\end{align*}
Since the above integral is also bounded below by zero, it follows that it is  $\mathcal{O}(R^{-N})$ and thus
\begin{equation}\label{eq:ineq1}
\left( 1+\epsilon \right)^{\alpha-N}  c(N, \alpha) R^{\alpha-N} -  \mathcal{O}(R^{-N}) < M(R)  \leq c(N, \alpha) R^{\alpha-N}.
\end{equation}

On the other hand, we may apply Tonelli's Theorem on the right hand side of \eqref{averagetogetunitmeasure} to obtain
\begin{align*}
\int_{B(R)}  \int_\amb  |x-y|^{\alpha-N} d\sigma_\alpha(y) d \tau_R(x)  &=  \int_\amb \int_{B(R)}    |x-y|^{\alpha-N} d \tau_R(x) d\sigma_\alpha(y) \\
& = \int_\amb U_\alpha^{\tau_R}(y) d\sigma_\alpha(y) \\
& = \int_{|y|\le R} U_\alpha^{\tau_R}(y) d\sigma_\alpha(y) + \int_{|y|>R} U_\alpha^{\tau_R}(y) d\sigma_\alpha(y).
\end{align*}
Applying the calculation of the potential in \eqref{eq:apot}, we find
\begin{equation}\label{eq:ineq2}
c(N, \alpha) R^{\alpha-N} \sigma_\alpha(B(R)) \leq M(R) < c(N, \alpha) R^{\alpha-N} \sigma_\alpha(\amb).
\end{equation}

Combining \eqref{eq:ineq1} and \eqref{eq:ineq2}, dividing by $R^{\alpha-N}$ and then letting $R\to\infty$, we obtain
\begin{displaymath}
\left( 1+\epsilon \right)^{\alpha-N}  \leq  \sigma_\alpha(\amb) \leq 1.
\end{displaymath}
Finally, we conclude $\sigma_\alpha(\amb) = 1$ by letting $\epsilon\to 0$.
\end{proof}

\begin{proof}[Proof of Theorem \ref{thm:2.3}]

For any positive Borel measure $\mu$, the potential
\begin{displaymath}
U_\alpha^{\mu}(t) = \int |t-x|^{\alpha-N} d\mu(x)
\end{displaymath}
is lower semicontinuous \cite[p.~59]{landkof}, and hence attains its infimum on the compact set $E$.  Thus we may choose $c_k \in E$ such that
\begin{displaymath}
\inf_E U_\alpha^{\nu_k} =U_\alpha^{\nu_k}(c_k)
\end{displaymath}
for each $k=1,\ldots, m$. It follows that
\begin{align*}
\sum_{k=1}^m \inf_E U_\alpha^{\nu_k} =& \sum_{k=1}^m U_\alpha^{\nu_k}(c_k) \\
=& \sum_{k=1}^m  \int |c_k-x|^{\alpha-N} d\nu_k(x)\\
\geq &   \int \min_{1\leq k \leq m} |c_k-x|^{\alpha-N} d\nu(x) \\
=& \int \left( \max_{1\leq k \leq m} |c_k-x| \right)^{\alpha-N} d\nu(x).
\end{align*}

The function $d_m(x) := \max_{1\leq k \leq m} |c_k-x|$ is the farthest distance function on the set of points $c_k$. By  Theorem \ref{thm:drep}, there exists a probability measure $\sigma_\alpha$ such that $U_\alpha^{\sigma_\alpha}(x) = d_m^{\alpha-N}(x)$. Applying Tonelli's Theorem, we have
\begin{displaymath}
\sum_{k=1}^m \inf_E U_\alpha^{\nu_k} \geq   \int U_\alpha^{\sigma_\alpha}(x) d\nu(x) = \int U_\alpha^\nu(t) d\sigma_\alpha(t).
\end{displaymath}
We estimate the potential $U_\alpha^\nu$ on $\amb$. Let $\mu_\alpha$ be the $\alpha$-equilibrium measure for $E$ and let $W_\alpha(E)$ be the $\alpha$-energy for $E$. Let $g(t) := U_\alpha^{\mu_\alpha}(t) - W_\alpha(E)$. By Frostman's Theorem \ref{thm:Frostman}, we know $g(t)\leq 0$ everywhere. On the other hand, $U_\alpha^\nu(t) - \inf_E U_\alpha^\nu \geq 0$ for $t\in E$. Thus
\begin{displaymath}
U_\alpha^\nu(t) \geq \inf_E U_\alpha^\nu + U_\alpha^{\mu_\alpha}(t) - W_\alpha(E)
\end{displaymath}
on $E$. It follows by the Principle of Domination \cite[Theorem 1.27 on p.~110 for $\alpha = 2$ and Theorem 1.29 on p.~115 for $0<\alpha<2$]{landkof} that this inequality holds in $\amb$. Thus, noting that $\sigma_\alpha$ is a unit measure and again applying Tonelli's Theorem, we find
 \begin{align*}
\sum_{k=1}^m \inf_E U_\alpha^{\nu_k} &\geq \int U_\alpha^\nu(t) d\sigma_\alpha(t) \\
&\geq \int \left( \inf_E U_\alpha^\nu  + \int U^{\mu_\alpha}(t) - W_\alpha(E) \right) d\sigma_\alpha(t) \\
&= \inf_E U_\alpha^\nu  + \int U^{\sigma_\alpha}_\alpha(x) d\mu_\alpha(x)- W_\alpha(E) \\
&= \inf_E U_\alpha^\nu  + \int  d_m^{\alpha-N}(x)d\mu_\alpha(x) - W_\alpha(E).
\end{align*}
By minimizing over all $m$-tuples $c_k$, we conclude that
\begin{displaymath}
\sum_{k=1}^m \inf_E U^{\nu_k} \geq C_E(\alpha, m) + \inf_E \sum_{k=1}^m  U^{\nu_k},
\end{displaymath}
where
\begin{displaymath}
C_E(\alpha, m) := \min_{c_k \in E} \int \min_{1\leq k\leq m} |x-c_k|^{2-N} d\mu_\alpha(x) - W_\alpha(E).
\end{displaymath}

We now show $C_E(\alpha, m)$ is the largest possible constant for a fixed $m$. We present two proofs of this fact. We begin with the shorter one which requires $E$ to be regular in the sense that $U^{\mu_\alpha}_\alpha(x) = W_\alpha(E)$ for all $x\in E$. Choose a set $c_k^*, \ k=1,\dots, m$, such that $\int  d_m^{\alpha-N}(x)d\mu_\alpha(x)$ attains its minimum on $E^m$. Let $d_m^*(x) := \min_{1\leq k\leq m} |x-c_k^*|$ and iteratively define the sets
\begin{align*}
S_1 =& \{x\in \supp(\mu_\alpha): |x-c_1^*| = d_m^*(x)\}, \\
S_k =& \{x\in \supp(\mu_\alpha) \setminus \cup_{j=1}^{k-1} S_j: |x-c_k^*| = d_m^*(x)\}, \ \text{$k=2,\ldots, m$.}
\end{align*}
It is clear that
\begin{displaymath}
\supp(\mu_\alpha) = \cup_{k=1}^m S_k \text{ and } S_k \cap S_j = \emptyset, \ k \neq j.
\end{displaymath}
Hence we can decompose $\mu_\alpha$ along the sets $S_k$ such that
\begin{displaymath}
\nu_k^* := \mu_\alpha |_{S_k} \text{ and } \mu_\alpha = \sum_{k=1}^m \nu_k^*.
\end{displaymath}
If $E$ is regular, then $\int |x-t|^{\alpha-N} d\mu_\alpha(t) = W_\alpha(E)$ for each $x \in E$ by Frostman's Theorem. Applying this fact, along with Tonelli's Theorem, we obtain
\begin{align*}
\sum_{k=1}^m \inf_E U_\alpha^{\nu_k^*} &\leq \sum_{k=1}^m U_\alpha^{\nu_k^*}(c_k^*) \\
&= \sum_{k=1}^m \int |c_k^*-x|^{\alpha-N} d\nu_k^*(x) \\
&= \sum_{k=1}^m \int (d_m^*(x))^{\alpha-N} d\nu_k^*(x) \\
&= \int (d_m^*(x))^{\alpha-N} d\mu_\alpha(x) \\
&= \int (d_m^*(x))^{\alpha-N} d\mu_\alpha(x) -W_\alpha(E) + \inf_{t\in E}\int |x-t|^{\alpha-N} d\mu_\alpha(x) \\
&= C_E(\alpha, m) + \sum_{k=1}^m \inf_E U_\alpha^{\nu_j^*}.
\end{align*}
Hence $C_E(\alpha, m)$ is sharp. The alternative proof uses points minimizing the discrete $\alpha$-energy and does not require that $E$ be regular. Let $\mathcal{F}_n = \{ \xi_{l, n} \}_{l=1}^n$ be the points of $E$ which minimize the discrete $\alpha$-energy. We will break the set $\mathcal{F}_n$ up using the points $c_k^*$ just as we broke up $\supp(\mu_\alpha)$ previously. Let $F_{k, n}$ be a subset of $\mathcal{F}_n$ such that $\xi_{l, n} \in \mathcal{F}_{l, n}$ if $d_m^*(\xi_{l, n}) = |\xi_{l, n}  - c_k^*|, \ 1\leq l \leq n$. If there is overlap between the sets, assign $\xi_{l, n}$ to only one set $\mathcal{F}_{k, n}$. It is clear that for any $n \in \mathbb{N}$,
\begin{displaymath}
\mathcal{F}_n = \cup_{k=1}^m \mathcal{F}_{k, n} \text{ and } \mathcal{F}_{k, n} \cap \mathcal{F}_{j, n} = \emptyset, \ k \neq j.
\end{displaymath}
Define the measures
\begin{displaymath}
\nu_{k, n}^* = \frac{1}{n} \sum_{\xi_{l, n} \in F_{k, n}} \delta_{\xi_{l, n}},
\end{displaymath}
so that for their potentials
\begin{displaymath}
p_{k, n}^*(x) = \frac{1}{n} \sum_{\xi_{l, n} \in F_{k, n}} |x-\xi_{l, n}|^{\alpha-N}, \ k=1,\ldots ,m,
\end{displaymath}
we have
\begin{displaymath}
\inf_E p_{k, n}^*(x) \leq \frac{1}{n} \sum_{\xi_{l, n} \in F_{k, n}} |c_k^*-\xi_{l, n}|^{\alpha-N} = \frac{1}{n} \sum_{\xi_{l, n} \in F_{k, n}} (d_m^*(\xi_{l, n}))^{\alpha-N}.
\end{displaymath}
It follows from the weak$^*$ convergence of $\nu_n := \sum_{k=1}^m \nu_{k, n}^* = \frac{1}{n} \sum_{l=1}^n \delta_{\xi_{l, n}}$ to $\mu_\alpha$, as $n\to \infty$, that
\begin{align*}
\limsup_{n\to \infty} \sum_{k=1}^m \inf_E p_{k, n}^*(x) &\leq \lim_{n\to\infty}\frac{1}{n} \sum_{k=1}^n (d_m^*(\xi_{k, n}))^{\alpha-N} \\
&= \int (d_m^*(x))^{\alpha-N} d\mu_\alpha(x).
\end{align*}
Applying Lemma \ref{lem:liminfFeketepts} to the potential $p_n^*$ of $\nu_n^*$ we find that
\begin{displaymath}
\lim_{n\to\infty}\inf_E p_n^*  = W_\alpha(E).
\end{displaymath}
It follows that
\begin{displaymath}
\limsup_{n\to\infty} \sum_{k=1}^m \inf_E p_n^* \leq C_E(\alpha, m) + \lim_{n\to\infty} \inf_E p_n^*.
\end{displaymath}
Hence we have asymptotic equality in \eqref{eq:infspotentials} as $n\to\infty$ with $m\geq 2$ being fixed, which shows that $C_E(\alpha, m)$ is the largest possible constant for each $m$. Since $d_m \leq d_E$ everywhere, we have $C_E(\alpha, m) \geq C_E(\alpha)$.
\end{proof}

\begin{proof}[Proof of Corollary \ref{cor:decreasingCEm}]
If $m< \card(\domset_E)$, then there is an $x_0 \in \supp(\mu_\alpha)$ such that $d_m^*(x_0) < d_E(x_0)$. As both functions are continuous, the same strict inequality holds in a neighborhood of $x_0$, so that $\int (d_m^*(x))^{\alpha-N} d\mu_\alpha(x) > \int d_E^{\alpha-N}(x) d\mu_\alpha(x)$ and hence $C_E(\alpha, m) > C_E(\alpha)$. This argument shows that if $\domset_E$ is infinite, then $C_E(\alpha, m) > C_E(\alpha)$ for $m\geq 2$. If $m\geq \card(\domset_E)$ then we may choose the points $c_k^*$ to include $\domset_E$ and hence $d_m^*(x) = d_E(x)$ for $x\in\supp(\mu_\alpha)$. Thus $C_E(\alpha, m) = C_E(\alpha)$.

Let $c_k$, $k=1,\ldots,m,$ be a set of points in $E$ that minimize the integral in the expression of $C_E(\alpha, m)$. Choose a point $c_{m+1}\in\partial E$. Then
\begin{align*}
C_E(\alpha, m) &= \int \min_{1\leq k\leq m} |x-c_k|^{\alpha-N} d\mu_\alpha(x) - W_\alpha(E) \\
&\geq \int \min_{1\leq k\leq m+1} |x-c_k|^{\alpha-N} d\mu_\alpha(x) - W_\alpha(E) \\
&\geq C_E(\alpha, m+1).
\end{align*}
Hence the constants $C_E(\alpha, m)$ are decreasing. It remains to show that their limit is $C_E(\alpha)$. Let $\{a_k\}_{k=1}^\infty$ be a countable dense subset of $E$. Then
\begin{displaymath}
C_E(\alpha) \leq C_E(\alpha,m) \leq \int \min_{1\leq k\leq m} |x-a_k|^{\alpha-N} d\mu_\alpha(x) - W_\alpha(E).
\end{displaymath}
Further, applying the Dominated Convergence Theorem, we have
\begin{align*}
\lim_{m\to\infty} \int \min_{1\leq k\leq m} |x-a_k|^{\alpha-N} d\mu_\alpha(x) &= \int \lim_{m\to\infty} \min_{1\leq k\leq m} |x-a_k|^{\alpha-N} d\mu_\alpha(x) \\
&= \int d_E^{\alpha-N}(x) d\mu_\alpha(x).
\end{align*}
The result follows.
\end{proof}

\begin{proof}[Proof of Corollary \ref{cor:smoothboundary}]
We show the minimal dominant set is infinite and then the result follows from Corollary \ref{cor:decreasingCEm}. Suppose to the contrary that $\domset_E=\{x_j\}_{j=1}^s$ is finite. Let $J\subset\partial E$ be a single connected component of the boundary. Define
\begin{displaymath}
J_k := \{x\in J : d_E(x) = |x-x_k|\}, k=1,\ldots , s.
\end{displaymath}
For each $x\in J_k$, the segment $[x, x_k]$ is orthogonal to $\partial E$ at $x_k$, by the smoothness assumption. Hence, each $J_k$ is contained in the normal line to $\partial E$ at $x_k$, $k=1,\ldots, s$. We thus obtain that $J=\cup_{k=1}^s J_k$ is contained in a union of straight lines which is a contradiction.
\end{proof}

\begin{proof}[Proof of Example \ref{ex:circle}]
To calculate the quantity $\min_{c_k \in \T} \int \min_{1\leq k\leq m} |x-c_k|^{\alpha-N} d\mu_\alpha(x)$, we follow an idea of Boyd \cite{boyd}. Let $c_k = -e^{i\psi_k},\ k=1,\ldots,m,$ with $\psi_k < \psi_{k+1},$ and for notational convenience let $\psi_0 = \psi_m$. Then we have $\max_{1\leq k\leq m} | e^{i\theta}-c_k| = |e^{i\theta} + e^{i\psi_k}| = |e^{i(\theta-\psi_k)} +1|$ for $\frac{\psi_{k-1} + \psi_k}{2} \leq \theta \leq \frac{\psi_k + \psi_{k+1}}{2}$ and hence
\begin{align*}
\int \min_{1\leq k\leq m} |x-c_k|^{\alpha-N} d\mu_\alpha(x) =& \frac{1}{2\pi} \sum_{k=1}^m \int_{\frac{\psi_{k-1} + \psi_k}{2}}^{\frac{\psi_{k} + \psi_{k+1}}{2}} |e^{i(\theta-\psi_k)} +1|^{\alpha-2} d\theta \\
=& \frac{1}{\pi} \sum_{k=1}^m \int_0^{\frac{\psi_{k} - \psi_{k-1}}{2}} |e^{i\theta} +1|^{\alpha-2} d\theta \\
=& \frac{2^{\alpha-2}}{\pi} \sum_{k=1}^m \int_0^{\frac{\psi_{k} - \psi_{k-1}}{2}} \cos^{\alpha-2}\left( \frac{\theta}{2} \right) d\theta \\
=& 2^{\alpha-2} \frac{2}{\pi} \sum_{k=1}^m I(\theta_k),
\end{align*}
where $\theta_k = \frac{\psi_{k} - \psi_{k-1}}{4}$ and $I(\theta_k) = \int_0^{\theta_k} \cos^{\alpha-2}(\theta) \ d\theta$. Since $I(\theta_k)$ is strictly convex for $0<\theta_k<\frac{\pi}{4}$, and $\sum_{k=1}^m \theta_k = \pi/2$, we have
\[
\frac{1}{m} \sum_{k=1}^m I(\theta_k) \ge I\left(\frac{\pi}{2m}\right).
\]
 Hence
\begin{displaymath}
\min_{c_k \in \T} \int \min_{1\leq k\leq m} |x-c_k|^{\alpha-N} d\mu_\alpha(x) = 2^{\alpha-2} \frac{2m}{\pi} I\left(\frac{\pi}{2m}\right),
\end{displaymath}
where the outer minimum is clearly attained for the equally spaced points $c_k$ on $\T.$
\end{proof}

\begin{proof}[Proof of Proposition \ref{circleprop}]
Equation (\ref{main}) is a consequence of (\ref{sphere}) and the main theorem proved in \cite{hks}.  The asymptotic formulas in (\ref{asym}) follow from (\ref{main}) and the asymptotics for $M_n^s(\mathbb{S}^1)$ given in \cite{hks}.
\end{proof}

 The proofs of Propositions \ref{sphereprop} and \ref{ballprop} are straightforward consequences of the main theorems on polarization proved in \cite{erdsaff},  \cite{borod} and \cite{BHS4}.

\medskip
{\bf Acknowledgements.} Research of I. E. Pritsker was partially supported by the U. S. National Security Agency under grant H98230-12-1-0227, and by the AT\&T Professorship. E. B. Saff was supported, in part, by the U. S. National Science Foundation under grant DMS-1109266. Research of W. Wise was done in partial fulfillment of the requirements for a PhD at Oklahoma State University, and with partial support from the AT\&T Professorship.

\emph{I. E. Pritsker}

Department of Mathematics,
Oklahoma State University,
401 Mathematical Sciences,
Stillwater, OK 74078-1058,
USA

igor@math.okstate.edu
\vspace{.1cm}

\emph{E. B. Saff}

Center for Constructive Approximation,
Department of Mathematics,
Vanderbilt~ University,
Nashville, TN 37240,
USA

edward.b.saff@vanderbilt.edu\
\vspace{.1cm}

\emph{W. Wise}

Department of Mathematics,
Oklahoma State University,
401 Mathematical Sciences,
Stillwater, OK 74078-1058,
USA

ninadawn@gmail.com


\begin{thebibliography}{1}

\bibitem{armitageandgardiner}
D.~H.~Armitage and S.~J.~Gardiner,
\newblock   Classical Potential Theory,
\newblock Springer-Verlag London Ltd., London, 2001.

\bibitem{aumann}
G.~Aumann,
\newblock Satz \"uber das verhalten von polynomen auf kontinuen,
\newblock   Sitz. Preuss. Akad. Wiss. Phys.-Math. Kl. (1933), 926--931.

\bibitem{blp2011}
A.~Baernstein, II, R.~S. Laugesen, and I.~E.~Pritsker,
\newblock Moment inequalities for equilibrium measures in the plane,
\newblock   Pure Appl. Math. Q. 7 (2011), 51--86.

\bibitem{borod}
S. V. Borodachov and N. Bosuwan,
\newblock Asymptotics of discrete Riesz $d$-polarization on subsets of $d$-dimensional manifolds,
\newblock   Potential Anal. (to appear).

\bibitem{BHS4}
S. V. Borodachov, D.P. Hardin and E.B. Saff,
\newblock Asymptotics of discrete Riesz polarization for rectifiable sets,
\newblock {(manuscript)}, 2013.

\bibitem{borwein}
P.~B.~Borwein,
\newblock Exact inequalities for the norms of factors of polynomials,
\newblock   Canad. J. Math. 46 (1994), 687--698.

\bibitem{boydII}
D.~W.~Boyd,
\newblock Two sharp inequalities for the norm of a factor of a polynomial,
\newblock   Mathematika 39 (1992), 341--349.

\bibitem{boyd}
D.~W.~Boyd,
\newblock Sharp inequalities for the product of polynomials,
\newblock   Bull. London Math. Soc. 26 (1994), 449--454.

\bibitem{erdsaff}
T.~Erd\'{e}lyi and E.~B.~Saff,
\newblock Riesz polarization inequalities in higher dimensions,
\newblock   J. Approx. Theory 171 (2013), 128--147.

\bibitem{farkas}
B. ~Farkas and B. ~Nagy,
 \newblock Transfinite diameter, Chebyshev constant and energy on locally compact spaces,
 \newblock   Potential Anal. 28 (2008) 241--260.

\bibitem{gelfond}
A.~O.~Gel'fond,
\newblock   Transcendental and Algebraic Numbers,
\newblock Dover Publications Inc., New York, 1960.

\bibitem{hks}
D.~P.~Hardin, A.~Kendall and E.~B.~Saff,
\newblock Polarization optimality of equally spaced points on the circle for discrete potentials,
\newblock   Discrete $\&$ Comp. Geometry 50 (2013), 236--243.
\newblock arXiv number: 1208.5261v1.

\bibitem{kneser}
H.~Kneser,
\newblock Das maximum des produkts zweies polynome,
\newblock   Sitz. Preuss. Akad. Wiss. Phys.-Math. Kl. (1934) 429--431.

\bibitem{kroopritsker}
A.~Kro{\'o} and I.~E.~Pritsker,
\newblock A sharp version of Mahler's inequality for products of polynomials,
\newblock   Bull. London Math. Soc. 31 (1999), 269--278.

\bibitem{landkof}
N.~S.~Landkof,
\newblock   Foundations of Modern Potential Theory,
\newblock Springer-Verlag, New York, 1972.

\bibitem{mahler}
K.~Mahler,
\newblock An application of Jensen's formula to polynomials,
\newblock   Mathematika 7 (1960), 98--100.

\bibitem{ohtsuka}
M.~Ohtsuka,
\newblock On various definitions of capacity and related notions,
\newblock   Nagoya Math. J. 30 (1967), 121--127.

\bibitem{pritskerspolynomials}
I.~E.~Pritsker,
\newblock Products of polynomials in uniform norms,
\newblock   Trans. Amer. Math. Soc. 353 (2001), 3971--3993.

\bibitem{reversetrianglegreen}
I.~E.~Pritsker,
\newblock Inequalities for sums of Green potentials and Blaschke products,
\newblock   Bull. Lond. Math. Soc. 43 (2011), 561--575.

\bibitem{prI}
I.~E.~Pritsker and S.~Ruscheweyh,
\newblock Inequalities for products of polynomials. {I},
\newblock   Math. Scand.  104 (2009), 147--160.

\bibitem{prII}
I.~E.~Pritsker and S.~Ruscheweyh,
\newblock Inequalities for products of polynomials. {II},
\newblock   Aequationes Math. 77 (2009), 119--132.

\bibitem{reversetriangle2D}
I.~E.~Pritsker and E.~B.~Saff,
\newblock Reverse triangle inequalities for potentials,
\newblock   J. Approx. Theory 159 (2009), 109--127.

\bibitem{me}
W.~Wise,
\newblock Potential theory and geometry of the farthest distance function,
\newblock   Potential Anal.  (2013) published online.

\end{thebibliography}
\end{document}